\documentclass{article}
\usepackage{graphicx} 
\usepackage{amsmath}
\usepackage{amsthm}
\usepackage{comment}
\usepackage{amsfonts}
\usepackage{mathrsfs}
\usepackage{authblk}

\newtheorem{theorem}{Theorem}
\newtheorem{proposition}{Proposition}
\newtheorem{lemma}{Lemma}

\newtheorem{remark}{Remark}

\usepackage{anysize}
\marginsize{2cm}{2cm}{2.5cm}{3cm}

\def\sE{{\mathscr{E}}}
\def\sA{{\mathscr{A}}}
\def\supp{\hbox{supp}\, }
\usepackage{color, xcolor}
\definecolor{wineRed}{rgb}{0.7,0,0.3}

\newcommand{\bR}{\mathbb R}
\newcommand{\cL}{\mathcal L}

\title{Asymptotic behavior of solutions to a space fractional diffusion  equation}
\author[{1}]{Barbara \L upi\'{n}ska\footnote{bpietruczuk@math.uwb.edu.pl}}
\author[{2}]{Piotr Rybka{\footnote{rybka@mimuw.edu.pl, corresponding author, ORCID: 0000-0002-0694-8201}}}

\affil[{1}]
{\small Faculty of Computer Science, University of Białystok, ul.Świerkowa 20B, 15-328 Białystok,  PL}
\affil[{2}]
{\small Faculty of 
Mathematics, Informatics and Mechanics\\ University of Warsaw, ul. Banacha 2, 02-097 
Warsaw, PL}
\date{\today}

\begin{document}

\maketitle
\begin{abstract}
\noindent We improve the time decay estimates of solutions to the one-dimensional fractional diffusion equation involving the Caputo derivative. The equation is considered on the half-line. Depending on the boundary condition, we show that solutions converge in $L^p$, $p>1$ to a multiple of the self-similar solutions or decay to zero. The convergence  rate is provided.
\end{abstract}

\section{Introduction}
It is well-known that solutions to the heat equation on the line with an initial condition decaying at infinity converge as time goes to infinity to a multiple of the fundamental solution, (see \cite[Theorem 1.1.4]{GGS}),
\begin{equation}\label{r1}
\lim_{t\to \infty} t^{\frac12}\| u(\cdot , t) - m G_t \|_{L^\infty} = 0,
\end{equation}
where $m = \int_\bR u(x,0)\, dx $ and $G_t(x) =\frac 12 (\pi t)^{-\frac12} \exp\big(- \frac{|x|^2}{4t}\big)$. 
The proof offered in \cite{GGS} is based on the representation of solutions in terms of convolution with the fundamental solution, $G_t$. 

In the present paper we prove an analogue of (\ref{r1}) for the space fractional diffusion
\begin{equation}\label{eq1}
\begin{array}{ll}
\frac{\partial u}{\partial t} - \frac{\partial}{\partial x} D^\alpha_{xC}u=0,     &  (x,t)\in(0,\infty)^2,\\
 u(x,0)=g(x),    & x>0,
\end{array}
\end{equation}
where $g,\ yg(\cdot)\in L^1(0,\infty)$ and 
$D^\alpha_{xC}$ is the fractional Caputo derivative of order $\alpha\in(0,1)$ with respect to the spacial variable $x$, see Section 2 for the definition of the operator $D^\alpha_{xC}$.

Problem \eqref{eq1} has to be augmented with the boundary condition at $x=0$. We choose the Dirichlet data
\begin{equation}\label{rD}
u(0, t) = 0,\qquad t>0
\end{equation}
or the fixed slope condition,
\begin{equation}\label{rN}
u_x(0, t) = 0,\qquad t>0.
\end{equation}

Here is our main result. Let us suppose that $w$ is a solution to \eqref{eq1}--\eqref{rD}, then Theorem \ref{t1} (a) states that
$$
\sup_{t\ge 1} t^{\frac{2-1/p}{1+\alpha}} \| w(\cdot , t) \|_{L^p}<\infty
$$
for all $p\in (1, \infty].$ We see that unexpectedly the Dirichlet boundary condition leads to extinguishing the solution. However, if $v$ is a solution to \eqref{eq1}, \eqref{rN}, then its behavior is consistent with \eqref{r1}, because we can show in  Theorem \ref{t1} (b) that
$$
\sup_{t\ge 1} t^{\frac{2-1/p}{1+\alpha}} \| v(\cdot , t) - 2m \sE_t \|_{L^p}<\infty
$$
for all $p\in (1, \infty].$ We see that the case $p=1$ is excluded. In fact, we can only prove weak type $(1,1)$ estimates, see Lemma \ref{l5}.

In the above formula we set $m = \int_0^\infty g(y)\,dy$ and $\sE_t$ is the fundamental solution of (\ref{eq1}$_1$), see Section 2 for the definition of $\sE_t$ and its properties. 

Our method of proving these facts is based on the representation formulas of solutions in terms of the convolution of data with  the fundamental solution. These formulas were derived in \cite{NRS}, we will recall them in Section \ref{s4}. However, in order to make them work we have to establish estimates on the decay of the fundamental solution to \eqref{eq1}, $\sE_t,$ and its derivatives, see Section \ref{s3} for a series of Lemmas. These bounds are an improvement in comparison with \cite{NRS} and they give us hope of lifting the restriction of the compact support of $g$. However, we need the boundedness of $\supp g$ in the proof of Theorem \ref{t1}.

The organization of the paper is as follows. 
In Section 2 we recall the basic facts from the fractional calculus. Section 3 is devoted to the proof of estimates on $\sE_t$ and its derivatives. The representation formulas, see \eqref{rn3.5} and \eqref{rn3.6} are recalled in Section 4, where the main theorem is proved with the help of Marcinkiewicz Interpolation Theorem.

\section{Preliminaries}
\subsection{Fractional operators}
First of all, we  recall the notions of the  Caputo and Riemann-Liouville fractional derivatives. We begin with the definition of the Riemann-Liouville integral and the derivative of fractional order $\alpha\in (0,1)$. Here they are 
$$[I^{\alpha}_{a+}f](x)=\frac{1}{\Gamma(\alpha)}\int_a^x{\frac{f(\tau)d\tau}{(x-\tau)^{1-\alpha}}},$$
$$D^{\alpha}_{xRL}f(x)=\frac{d}{dx}[I^{1-\alpha}_{a+}f],$$
while the Caputo fractional derivative of order $\alpha$ has the form
$$D^{\alpha}_{xC}f(x)=[I^{1-\alpha}_{a+}f'](x)=\frac{1}{\Gamma(1-\alpha)}\int_a^x{\frac{f'(\tau)d\tau}{(x-\tau)^{\alpha}}}.$$
\begin{remark}[see \cite{Kilbas}]
The operator $I^{\alpha}_{a+}f$ is well defined for $f\in L^p(a,b),\ p\geq 1$ and $D^{\alpha}_{xRL}f,\ D^{\alpha}_{xC}f$ are well defined, for example, for an absolutely continuous functions. This approach is sufficient for our purposes.
\end{remark}

The following formula illustrates the relationship between the two types of derivatives for a general function $f\in AC[0,L]$,
$$D^{\alpha}_{xRL}f(x)=D^{\alpha}_{xC}f(x)+\frac{x^{-\alpha}}{\Gamma(1-\alpha)}f(0).$$
Let us notice that if $f(0)=0$, then 
$$D^{\alpha}_{xRL}f(x)=D^{\alpha}_{xC}f(x).$$
Moreover, if $\alpha\in (0,1)$, then  we have
$$D^{\alpha}_{xRL}I^\alpha _{a+}f(x)=f(x),$$
but
\begin{equation}\label{r-ID}
I^\alpha _{a+}D^{\alpha}_{xRL}f(x)=f(x)-f(a)    ,
\end{equation}
for absolutely continuous 
functions. 
For more details about fractional operators, we refer to~\cite{Kilbas}.

\subsection{The fundamental solution}
The authors of \cite{NRS} derived a formula for $\sE$, a self-similar solution to \eqref{eq1}. Namely, its definition follows
\begin{equation}
       \sE_t(x) = \sE(x,t)=a_0t^{-\frac{1}{1+\alpha}}E_{\alpha,1+1/\alpha,1/\alpha}\left(-\frac{x^{1+\alpha}}{(1+\alpha)t}\right),
   \end{equation}
where
$$
E_{\alpha, 1+1/\alpha, 1/\alpha} (z)= \sum_{n=0}^\infty  b_n z^n
$$
and
\begin{equation*}
    b_n = \prod_{i=0}^{n-1} \frac{\Gamma(\alpha i + i +2)}{\Gamma(\alpha(i+1) + i +2)}, \quad n>0, \qquad b_0=1.
\end{equation*}
In general, 
$E_{\alpha, \beta, \gamma}$ is the three parameter Mittag-Leffler function, see \cite[formula (1.9.19)]{Kilbas} for the definition. In particular, $E_{\alpha, \beta, \gamma}$ is an entire function.

The constant $a_0$ is determined by the condition
\begin{equation*}
\frac{1}{a_0}=2\int_0^\infty E_{\alpha,1+1/\alpha,1/\alpha}\left(-x^{1+\alpha}/(1+\alpha)\right)dx.
\end{equation*}  
Since $\sE$ plays the role of the fundamental solution and it is used to represent solutions, see Proposition \ref{p1} below, we collect here its main properties.  They are shown in \cite[Theorem 1]{NRS}:\\
\begin{proposition}\label{p0} Let us suppose $\sE$ is the fundamental solution defined above. Then:
\\(1) $\sE\in C^2(0,\infty)^2$;\\
    (2) $\sE$ is positive for all $x,\ t>0$;\\
    (3) for all $t>0$ function $\sE(\cdot,t)$ is decreasing;\\
    (4) for all $t>0$ function $\sE(\cdot,t)$ is in $L^1(0,\infty)$ and $\displaystyle{\int_0^\infty \sE(x,t)dx=\frac{1}{2}.}$ 
\end{proposition}
    
Let us set 
\begin{equation}\label{fi}
\Phi(x)= E_{\alpha,1+1/\alpha,1/\alpha}\left(-\frac{x^{1+\alpha}}{1+\alpha}\right),
\end{equation}
we will call it the profile function.
We can write $\sE$ shortly as follows
\begin{equation}\label{df-e}
\sE_t(x)=
a_0t^{-\frac{1}{1+\alpha}}\Phi\left(\frac{x}{t^{\frac{1}{1+\alpha}}}\right).
\end{equation}

What is important for the derivation of the properties of $\Phi$ is not the series defining the Mittag-Leffler function, but the differential equation $\Phi$ satisfied. 
We recall:
\begin{lemma}(see \cite[Proposition 1]{NRS})
\label{L_rown}\\
The function $\Phi$ defined in~\eqref{fi} satisfies the following fractional differential equation, when $\alpha\in (0,1)$
\begin{equation}\label{r2}
    D^\alpha_{xC}v(x)=-\frac{1}{1+\alpha} xv(x),\;\; x>0,\;\; v(0)=1.
\end{equation}\end{lemma}

\section{The  decay properties of the profile function $\Phi$}\label{s3}
The goal of this subsection is to establish a number of estimates for solutions to~\eqref{r2}. First, we recall a known result.

\begin{lemma}\label{l-st}(see \cite[Lemma 8]{NRS})
    Let us suppose that $\Phi$ given by~\eqref{fi}. Then, for all $x>0$ we have
    $$0<x\Phi(x)\leq 4.$$
\end{lemma}
Let us stress that for $x\in[0,1]$ the above inequality is far from being useful. Indeed, 
monotonicity of $\Phi$ guaranteed by Proposition \ref{p0} (3) and $\Phi(0)=1$ imply that $\Phi(x)\le 1$ for $x\in[0,1]$. 

Now, we will present step-by-step improvements of the above estimate. The goal is to show how fast $\Phi$ and $\Phi'$ decay at infinity. 
Here is our first observation:
\begin{lemma}\label{L_oszacowanieFi}
  For all $x>1$, the profile function $\Phi$ defined in~\eqref{fi} satisfies the following inequality
\begin{equation}\label{oszacowanieFi}
\Phi(x)\leq\Gamma(\alpha+2)\left(\frac{2}{x}\right)^{\alpha+1}.
\end{equation}
\end{lemma}
\begin{proof}
    From Lemma~\ref{L_rown}, we have
\begin{equation}\label{eq4}
    D^\alpha_{xC}\Phi(x)=-\frac{1}{1+\alpha} x\Phi(x),\  x>0,\qquad  \Phi(0)=1.
\end{equation}
After we apply the fractional  Riemann-Liouville integral operator, $I^{\alpha},$ to both sides of~\eqref{eq4} and take into account \eqref{r-ID}, we obtain
$$\Phi(x)=1-\frac{1}{\Gamma(\alpha)(1+\alpha)}\int_0^x\frac{s\Phi(s)ds}{(x-s)^{1-\alpha}}\geq 0,\qquad x>0.$$
Therefore, using the monotonicity of the function $\Phi$, we get
$$\Gamma(\alpha)(1+\alpha)\geq\int_0^x\frac{s\Phi(s)ds}{(x-s)^{1-\alpha}}=\int_0^\frac{x}{2}\frac{s\Phi(s)ds}{(x-s)^{1-\alpha}}+\int_\frac{x}{2}^x\frac{s\Phi(s)ds}{(x-s)^{1-\alpha}}$$
$$\geq \int_\frac{x}{2}^x\frac{s\Phi(s)ds}{(x-s)^{1-\alpha}}\geq\frac{x}{2}\Phi(x)\int_\frac{x}{2}^x(x-s)^{\alpha-1}ds=\frac{1}{\alpha}\left(\frac{x}{2}\right)^{\alpha+1}\Phi(x),$$
which ends the proof.
\end{proof}

After establishing an improved decay estimate for $\Phi$ we are in a position to prove an  estimate for $\Phi'$.
The lemma below plays a crucial role in  estimating  the solutions to \eqref{eq1}.
\begin{lemma}\label{l4}
For all $x>0$, the derivative of the profile function $\Phi$ can be 
estimated as follows,
\begin{equation}\label{poch}
|\Phi'(x)|\leq 
\left\{
\begin{array}{cc}
 \frac1{\Gamma(1+\alpha)} x^\alpha    &  x\in [0, 1], \\
 \frac{C_1(\alpha)}{x}&  x>1,
\end{array}
\right.
\end{equation}
where 
$$
C_1(\alpha) = \alpha 2^{1+\alpha}\max\bigg\{\int_{\frac12}^1 \frac{du}{ u^{1+\alpha}(1-u)^{1-\alpha}}, 2\bigg\}.
$$
\end{lemma}
\begin{proof}
We begin with the fractional differential equation presented in Lemma~\ref{L_rown}
\begin{equation}
    D^\alpha_{xC}\Phi(x)=-\frac{1}{1+\alpha} x\Phi(x),\;\; x>0,\;\; \Phi(0)=1.
\end{equation}
After we apply  the Riemann-Liouville derivative $D^{1-\alpha}_{xRL}$ to both sides, 
we obtain
$$\Phi'(x)=-\frac{1}{1+\alpha}D^{1-\alpha}_{xRL}[x\Phi(x)]=-\frac{1}{1+\alpha}D^{1-\alpha}_{xC}[x\Phi(x)]=-\frac{1}{1+\alpha}I^{\alpha}[x\Phi(x)]'<0.$$
Using the definition of the Riemann-Liouville integral, we have
$$0<I^{\alpha}[x\Phi(x)]'=\frac{1}{\Gamma(\alpha)}\int_0^x{\frac{(\Phi(s)+s\Phi '(s))ds}{(x-s)^{1-\alpha}}},\;\;\;x>0.$$
Of course, due to monotonicuty of $\Phi$ we have
$$0< \Phi(x)+x\Phi '(x)=\Phi(x)-x|\Phi '(x)|,\quad x>0.$$
If $x\in [0,1]$, then the above inequality implies that
$$
| \Phi'(x)| \le \frac{1}{\Gamma(\alpha)(1+\alpha)}\int_0^x{\frac{\Phi(s)ds}{(x-s)^{1-\alpha}}} \le \frac{\alpha}{\Gamma(2+\alpha)}\int_0^x\frac{1 }{(x-s)^{1-\alpha}}
\, ds = \frac{x^\alpha}{\Gamma(2+\alpha)}.$$

Now, for $x>1$, we proceed differently. We have
$$
|\Phi'(x)| = - \Phi'(x) =
\frac\alpha{\Gamma(2+\alpha)} \bigg(\int_0^{ax} + \int_{ax}^x\bigg) \frac{(s\Phi'(s)}{(x-s)^{1-\alpha}}\, ds=: I_1 + I_2,
$$
where $a\in (0,1)$ is arbitrary and we will take $a=\frac12$. In order to estimate $I_2$ we take into account that $\Phi'(x)<0$ and we use Lemma~\ref{L_oszacowanieFi},
\begin{align*}
\frac{\Gamma(2+\alpha)}\alpha I_2 &=  \int_{ax}^x 
\frac{\Phi(s) + s \Phi'(s)}{(x-s)^{1-\alpha}}\, ds
\le \int_{ax}^x 
\frac{ \Phi(s) }{(x-s)^{1-\alpha}}\, ds \\
&\le  \int_{ax}^x 
\frac{ \Gamma(2+\alpha) 2^{1+\alpha}\, ds }{s^{1+\alpha}(x-s)^{1-\alpha}} 
= \frac {\Gamma(2+\alpha) 2^{1+\alpha}}x \int_a^1 \frac{du}{ u^{1+\alpha}(1-u)^{1-\alpha}}.\\
\end{align*}
Hence, for $a= \frac12$ we have
$$
I_2 \le \alpha 2^{1+\alpha} \int_{\frac12}^1 \frac{du}{ u^{1+\alpha}(1-u)^{1-\alpha}} \frac 1 x.
$$
In order to estimate $I_1$ we will integrate by parts
\begin{align*}
\frac{\Gamma(2+\alpha)}\alpha I_1 &=  (\alpha-1)\int^{ax}_0 \frac{s\Phi(s)}{(x-s)^{2-\alpha}} + \frac{s\Phi(s)}{(x-s)^{1-\alpha}}\bigg|_{s=0}^{s=ax}\le \frac{ax \Phi(ax)}{(1-a)^{1-\alpha}x^{1-\alpha}}.
\end{align*}
We dropped the integral because it is negative. Finally, Lemma \ref{L_oszacowanieFi} yields for $a=\frac12$,
$$
I_1 \le \frac{\alpha 2^{2+\alpha}} x
$$
and
$$
|\Phi'(x)| \le C_1(\alpha) \frac 1x,
$$
where $ C_1(\alpha) = \alpha 2^{1+\alpha}\max\{\int_{\frac12}^1 \frac{du}{ u^{1+\alpha}(1-u)^{1-\alpha}}, 2\}. $
Our claim follows.
\end{proof}
\begin{remark}
Certainly, the above estimate is far from being optimal, since integration of the estimate of $\Phi'$ provided above yields a logarithmic bound on $\Phi$ which is much worse than the result of Lemma \ref{L_oszacowanieFi}.
\end{remark}

\section{The convergence result} \label{s4}
In this section, we show our convergence result. 
For this purpose we recall the representation formulas.

\begin{proposition}(see \cite[Theorem 4.1]{NRS})\label{p1}
Let us extend $\sE$  to   $\bR^2$ by the following formula $$\sE(x,t) = \sE(|x|,t) \chi_{\bR_+}(t), \qquad (x,t) \in \bR^2.$$
Let us suppose that  $g\in L^p(0,\infty),$ where  $p\in[1,\infty)$ and $g$ has compact support, (resp. $g\in C^0_c([0,\infty))$
and   we set $
\sE_t(x) = \sE(x,t).$
We define functions $w_1$, $w_2$ by the  following formulas,
\begin{equation}\label{rn3.5}
w_1(x,t) = \int_0^\infty ( \sE_t(x-y) - \sE_t(x+y))g(y)\,dy,
\end{equation}
\begin{equation}\label{rn3.6}
w_2(x,t)  = \int_0^\infty ( \sE_t(x-y) + \sE_t(x+y))g(y)\,dy.
\end{equation}
Then, \\
(a) For all $t>0$ functions $w_1(\cdot,t)$ and $w_2(\cdot,t)$ belong to $C^{1+\alpha}_{loc}(\bR_+)$. Moreover, $w_1$ (resp. $w_2$) is a classical solution to \eqref{eq1}, \eqref{rD} (resp. \eqref{eq1}, \eqref{rN}).\\
(b) The functions $w_1$ and $w_2$ belong to $L^\infty(\bR_+; L^p(\bR_+))$ if $p<\infty$ (resp. $w_1, w_2 \in L^\infty(\bR_+; C(\bR_+))$, when $g\in C^0_c([0,\infty))$).
Moreover,
\begin{equation*}
 \lim_{t\to 0}\|w_1(\cdot, t) - g\|_{L^p} =0, \qquad(\hbox{resp. }
\lim_{t\to 0}\|w_2(\cdot, t) - g\|_{L^\infty} =0).   
\end{equation*}
When $g$ is continuous, we require $g(0) =0$ in case of the Dirichlet boundary condition, \eqref{rD}.
\end{proposition}
Our asymptotic analysis depends in a crucial manner  on this representation formula. Here is the result which follows from \eqref{rn3.5}  and \eqref{rn3.6}.
\begin{theorem}\label{t1}
Let us suppose that  $g,\ yg(\cdot) \in L^1(0,\infty)$ have a compact support, $\theta$ is any fixed number from the interval $(0,1)$, 
\begin{equation}\label{df-m}
m=\int_0^\infty g(y)dy,
\end{equation} 
and $w_1$, (resp. $w_2$), is defined by~\eqref{rn3.5}, (resp.~\eqref{rn3.6}).
Then, \\
(a) there is $C_1>0$ such that 
$$
\|w_1(\cdot,t) \|_{L^{1/(1-\theta)}} \le C_1  t^{- (1+\theta)/(1+\alpha)} \|y g(\cdot)\|_{L^1}, \qquad t>1;
$$
(b) there is $C_2>0$ such that 
$$
\|w_2(\cdot,t)-2m\sE_t(\cdot)\|_{L^{1/(1-\theta)}} \le C_2 
 t^{- (1+\theta)/(1+\alpha)} \|y g(\cdot)\|_{L^1},\qquad t>1;
$$
(c) for $C_2(\alpha)$ defined in Lemma \ref{l5} we have
$$
\|w_2(\cdot,t)-2m\sE_t(\cdot)\|_{L^\infty}, \quad 
\|w_1(\cdot,t) \|_{L^\infty} \le \frac{C_2(\alpha)}{t^{2/(1+\alpha)}} \|y g(\cdot)\|_{L^1}, \qquad t>1.
$$
\end{theorem}
Before we start the proof we notice  that the two key objects we have to study are
\begin{equation}\label{df-v}
v^+(x,t) = \int_0^\infty    \sE_t(x+y)g(y)\,dy,\qquad
v^-(x,t) = \int_0^\infty \sE_t(x-y) g(y)\,dy.
\end{equation}
Here is our main observation.
\begin{lemma}\label{l5}
Let us suppose that $g$ satisfies the assumptions of Theorem \ref{t1}, $m$ is given by \eqref{df-m} and $v^+$, $v^-$ are defined in \eqref{df-v}. Then, 
$$
|v^\pm (x, t)  - m \sE_t(x)| \le \frac{C_2(\alpha)}{t^{1/(1+\alpha)}}\| y g(\cdot)\|_{L^1}\bigg(\frac1{t^{1/(1+\alpha)}}\chi_{[0, t^{1/(1+\alpha)}]}(x) + \frac 1x \chi_{(t^{1/(1+\alpha)}, \infty)}(x)\bigg)\qquad\hbox{for }t>1,
$$
where $C_2(\alpha) =2a_0 C_1(\alpha)$ and $C_1(\alpha)$ is defined in Lemma \ref{l4}.
\end{lemma}
\begin{proof} We will first estimate $v^+ - \sE_t$.
Let us notice that due to the definitions of $\sE_t$ and $m$, see \eqref{df-e} and \eqref{df-m}, we have
\begin{align*}
v^+(x,t) - m\sE_t(x) &= \frac{a_0}{t^{1/(1+\alpha)}}
\int_0^\infty
\bigg( \Phi \bigg(\frac{x+y}{t^{1/(1+\alpha)}}\bigg)
-  \Phi \bigg(\frac{x}{t^{1/(1+\alpha)}}\bigg)\bigg)g(y)\, dy\\
&= \frac{a_0}{t^{1/(1+\alpha)}}
\int_0^\infty \int_0^1 \frac{y}{t^{1/(1+\alpha)}}
\Phi'\bigg(\frac{x+sy}{t^{1/(1+\alpha)}}\bigg)\, ds g(y)\, dy =:I
.
\end{align*}
We consider two cases: $\frac{x}{t^{1/(1+\alpha)}}< 1$ and  $\frac{x}{t^{1/(1+\alpha)}} \ge 1$. 

When $\frac{x}{t^{1/(1+\alpha)}}< 1$ we set $y_0 = t^{1/(1+\alpha)}- x$ and $\xi=\frac{x+sy}{t^{1/(1+\alpha)}}$. We calculate
$$
I = \frac{a_0}{t^{2/(1+\alpha)}} \bigg(\int_0^{y_0} + \int_{y_0}^\infty\bigg)\int_0^1 y \Phi'\big(\xi
\big)\, ds g(y)\, dy = I_1 + I_2.
$$
We first estimate $|I_1|$ by noticing that $|\xi|\le1$ and $x+y_0= t^{1/(1+\alpha)}$,
\begin{align*}
 |I_1| &\le \frac{a_0}{\Gamma(1+\alpha)t^{2/(1+\alpha)}} \int_0^{y_0} \int_0^1y| \xi|^\alpha
\, ds |g(y)|\, dy
\le \frac{a_0}{\Gamma(1+\alpha)t^{2/(1+\alpha)}} 
\int_0^{y_0} y| g(y)| \, dy,   
\end{align*}
where we applied  Lemma \ref{l4} to bound $\Phi'.$

Now, we estimate $I_2$  and again we will use Lemma \ref{l4},
\begin{align*}
 |I_2 |&\le 
 \frac{a_0 C_1(\alpha)}
 {t^{2/(1+\alpha)}} \int_{y_0}^\infty \bigg(
 \int_0^{s_0(y)} \bigg(\frac{x+sy}{t^{1/(1+\alpha)}}\bigg)^\alpha \, ds + \int_{s_0(y)}^1 
 \frac {t^{1/(1+\alpha)}}{x+sy}\,ds \bigg)y|g(y)|\, dy\\ &
 \le \frac{a_0 C_1(\alpha)}{
 t^{2/(1+\alpha)}} \int_{y_0}^\infty y |g(y)|\, dy,
\end{align*}
where we set $s_0(y) = (t^{1/(1+\alpha)}-x)/y$ for $y\ge y_0$. 
Hence,
$$
|I| \le \frac{a_0 C_1(\alpha)}{
t^{2/(1+\alpha)}} \int_{0}^\infty y |g(y)|\, dy
$$
for $x< t^{1/(1+\alpha)}$.

Now, we take care of the second case, i.e. $\frac{x}{t^{1/(1+\alpha)}} \ge 1$. If this happens, then the argument of $\Phi'$ is always greater than 1. Hence, we proceed as in the estimate of $I_2$ above, however, $y_0=0$. Thus,
$$
|I| \le \frac{a_0 C_1(\alpha)}{
xt^{1/(1+\alpha)}} \int_{0}^\infty y |g(y)|\, dy.
$$
Our claim follows for $v^+$.


Now, we turn our attention to $v^- -m\sE_t$. We notice
$$
v^-(x,t) - m\sE_t (x)= \frac{a_0}{t^{1/(1+\alpha)}}\int_0^\infty
\bigg(\Phi\bigg( \frac{x- y}{t^{1/(1+\alpha)}}\bigg)
-\Phi\bigg( \frac{x}{t^{1/(1+\alpha)}}\bigg)\bigg)g(y)\,dy.
$$
We also see that $\Phi$ extended to negative arguments is differentiable at $0$ due to \eqref{poch}. Hence,
$$
v^-(x,t) - m\sE_t(x) = -\frac{a_0}{t^{2/(1+\alpha)}}\int_0^\infty
\int_0^1 y \Phi'\bigg(\frac{x-sy}{t^{1/(1+\alpha)}}\bigg)\,ds g(y)\, dy=:I.
$$
We must take into account the possibility that  the argument of the integrand  is negative. We recall that our datum $g$ has a compact support, we may assume that $\supp g\subset[0,R]$. Since we are interested in the time assymptotics we may consider only $t> R^{1+\alpha}.$
We will come up with different estimates for $|v^-(x,t)-m \sE_t(x)|$ when $x<2R$ and $x\ge 2R$.

Let us consider $x<2R$. Then we notice
$$
\bigg| \Phi'\bigg(\frac{x-sy}{t^{1/(1+\alpha)}}\bigg)\bigg|\le C_1(\alpha).
$$
Indeed, if $\xi$ denotes the argument of $\Phi'$, then
Lemma \ref{l4} yields $|\Phi'(\xi)| \le \frac{|\xi|^\alpha}{\Gamma(1+\alpha)}\le 1$, when  $|\xi|\le 1$ and $|\Phi'(\xi)| \le \frac {C_1(\alpha)}{|\xi|}\le C_1(\alpha)$ for $|\xi|>1$. Hence
$$
|v^-(x,t)-  m \sE_t(x)|\le C_1(\alpha), \qquad\hbox{for }x\le \max\{2R, t^{1/(1+\alpha)}\}.
$$
When $x>\max\{2R, t^{1/(1+\alpha)}\}$, then for $s\in [0,1]$ and $y\in[0,R]$ we have
$|x-sy|/t^{1/(1+\alpha)}>1$. Hence,
$$
\bigg| \Phi'\bigg(\frac{x-sy}{t^{1/(1+\alpha)}}\bigg)\bigg|\le
t^{1/(1+\alpha)} \frac {C_1(\alpha)}{x-sy}.
$$
As a result, we have
\begin{align*}
|I| & \le  \frac{a_0 C_1(\alpha)}{t^{1/(1+\alpha)}}\int_0^\infty
\int_0^1  \frac y{x-sy}\,ds |g(y)|\, dy
\le    
\frac{a_0 C_1(\alpha)}{t^{1/(1+\alpha)}}\int_0^\infty \frac y{x-y} |g(y)|\,dy \\ &
\le  \frac{ 2a_0 C_1(\alpha)}{t^{1/(1+\alpha)}}\int_0^\infty \frac y{x} |g(y)|\,dy ,
\end{align*}
where we also used $1/(x-y) \le 2/x$ for $x> 2R.$
\end{proof}

We are now ready for {\it the proof of Theorem \ref{t1}.} 
We are going to establish the estimates for  $v^\pm - m \sE_t$.
For this purpose we will use the above Lemma and the following Marcinkiewicz Interpolation Theorem:
\begin{proposition}(see \cite[Theorem 6.0.2]{simon}, \cite[Theorem 1]{zygmund})\label{t-im}\\
Let $(M.\Sigma,\mu)$ and $(N,\Xi,\nu)$ be two $\sigma$-finite measure spaces. Let $1\le p_0\le q_0\le\infty$, $1\le p_1\le q_1\le \infty$ and 
$T:L^{p_0}(M, \mu)+L^{p_1}(M,\mu)$ into the $\Xi$-measurable functions on $N$ so that $T$ is of the weak  types $(p_0,q_0)$ and  $(p_1,q_1)$ with constants $C_0$ and $C_1$, respectively. Then, for any $\theta\in (0,1)$ there is a constant $C_\theta$ depending only on $\theta,$ $p_j$, $q_j$, $C_j$, $j=0,1$, so that
$$
\| Tf \|_{q_\theta} \le K C_0^{1-\theta}C_1^\theta \| f \|_{p_\theta},
$$
where 
$$
\frac 1{p_\theta} = \frac\theta{p_1} + \frac{1-\theta}{p_0},\qquad \frac 1{q_\theta} = \frac\theta{q_1} + \frac{1-\theta}{q_0}
$$
and $K$ is independent of $f$, it is bounded if $p_0$, $p_1,$ $q_0$, $q_1$ are fixed and $\theta$ stays away from $0$ and $1$.
\end{proposition}

Lemma \ref{l5}
gives us two types of results. First we note that
$$
(M,\Sigma, \mu ) = (\bR_+, \sA, y\cL^1), \qquad
(N,\Sigma, \nu ) = (\bR_+, \sA, \cL^1), 
$$
where $\sA$ is the $\sigma$-field of Lebesgue measurable sets, $\cL^1$ is the one-dimensional Lebesgue measure. For a fixed $t>0$ we set
$T^\pm g = v^\pm(t, \cdot) - m\sE_t$. For a fixed $\theta$ we set
$$
p_0 = 1 = p_1,\qquad q_0=1,\ q_1 = \infty.
$$
We claim that $T^\pm$ are of type $(1,\infty)$ with constant $C_1 = C_2(\alpha) t^{-2/(1+\alpha)}$, i.e.,
$$
\| v^\pm -m \sE_t\|_{L^\infty} \le C_2(\alpha) t^{-2/(1+\alpha)} \| y g(\cdot)\|_{L^1} = C_2(\alpha) t^{-2/(1+\alpha)} \| g\|_{L^1_\mu}.
$$
Indeed, for $x<t^{1/(1+\alpha)}$, we have
$$
| v^\pm(x,t) -m \sE_t(x)| \le C_2(\alpha) t^{-2/(1+\alpha)} \|  g \|_{L^1_\mu}.
$$
For $x\ge t^{1/(1+\alpha)}$ we have
$$
| (T^\pm g)(x)| \le C_2(\alpha) t^{-1/(1+\alpha)} x^{-1}\| y g(\cdot)\|_{L^1}\le C_2(\alpha) t^{-2/(1+\alpha)} \|  g\|_{L^1_\mu}.
$$

We are going to check that the mapping $T^\pm $ is of the weak type $(1,1)$, i.e. we will check that
$$
\sup_{\sigma>0} \sigma | \{ x: |T^\pm g(x)|>\sigma\}| \le C_2(\alpha) t^{-1/(1+\alpha)} \|g\|_{L^1_\mu}.
$$
Indeed, this immediately follows from Lemma \ref{l5}, moreover $C_0 = C_2(\alpha)  t^{-1/(1+\alpha)} $.

Now, we are going to use the Marcinkiewicz Interpolation Theorem, Proposition \ref{t-im}, with 
$$
p= 1, \qquad q= 1/(1-\theta)
$$
and 
$$
C_0^{1-\theta} C_1^\theta = t^{- (1+\theta)/(1+\alpha)}.
$$
Hence, for $\theta\in (0,1)$ we have
$$
\| v^\pm(x,t) -m \sE_t(x)\|_{L^{1/(1-\theta)}} \le K C_2(\alpha) t^{- (1+\theta)/(1+\alpha)},
$$
where $K$ is finite for $\theta$ different from $0$ and $1.$ Thus,
\begin{align*}
\| w_1 (\cdot, t)\|_{L^{1/(1-\theta)}}& \le 
\|v^+(x,t) -m \sE_t(x) \|_{L^{1/(1-\theta)}} + \|v^-(x,t) -m \sE_t(x) \|_{L^{1/(1-\theta)}} \\ &
\le 2KC_2(\alpha) t^{- (1+\theta)/(1+\alpha)} \| g\|_{L^1_\mu}.   
\end{align*}
Moreover,
\begin{align*}
\|w_2(\cdot,t)-2m\sE_t(\cdot)\|_{L^{1/(1-\theta)}}& \le 
\|v^+(x,t) -m \sE_t(x) \|_{L^{1/(1-\theta)}} + \|v^-(x,t) -m \sE_t(x) \|_{L^{1/(1-\theta)}} \\ &
\le 2K C_2(\alpha) t^{- (1+\theta)/(1+\alpha)} \| g\|_{L^1_\mu}.
\end{align*}
Part (c) follows immediately from Lemma \ref{l5}. \qed
\begin{remark}
The boundedness of the support of $g$ is used in the proof of the above theorem, when we consider $t> \max \supp g.$ It is not clear if this assumption is essential.
\end{remark}

\end{document}